\documentclass[12pt, reqno]{amsart}

\usepackage{amsmath, amsthm, amssymb}
\usepackage{enumerate, eucal}
\usepackage{amsfonts}
\usepackage{latexsym}
\usepackage{epsfig}

\numberwithin{equation}{section}

\def \ker {{\rm ker}\,}

\def \dim {{\rm dim}\,}

\newcommand{\mf}{\mathfrak}
\newcommand{\m}{\CMcal}

\theoremstyle{plain}
\newtheorem{theorem}{Theorem}[section]

\newtheorem{proposition}[theorem]{Proposition}

\newtheorem{definition}[theorem]{Definition}

\newtheorem{remark}[theorem]{Remark}
\newtheorem{example}[theorem]{Example}

\begin{document}
	\title[Quantum $U$-channels on $S$-spaces]{Quantum $U$-channels on $S$-spaces}
	\date{\today}
	
	\date{\today}
		\author[Bag]{Priyabrata Bag}
	\address{Department of Mathematics, School of Science, Gandhi Institute of Technology and Management (Deemed-to-be University), Doddaballapur Taluk, Bengaluru, Karnataka 561203, INDIA}
	\email{pbag@gitam.edu, priyabrata.bag@gmail.com}
	
	\author[Rohilla]{Azad Rohilla\textsuperscript{*}}
	
	\address{Centre for Mathematical and Financial Computing, Department of Mathematics, The LNM Institute of Information Technology, Rupa ki Nangal, Post-Sumel, Via-Jamdoli
		Jaipur-302031,
		(Rajasthan) INDIA}
	
	\email{18pmt005@lnmiit.ac.in, azadrohilla23@gmail.com}
	\thanks{*corresponding author}
	
	\author[Trivedi]{Harsh Trivedi}
	\address{Centre for Mathematical and Financial Computing, Department of Mathematics, The LNM Institute of Information Technology, Rupa ki Nangal, Post-Sumel, Via-Jamdoli
		Jaipur-302031,
		(Rajasthan) INDIA}
	\email{harsh.trivedi@lnmiit.ac.in, trivediharsh26@gmail.com}


	\begin{abstract}
	If the symmetry, (an operator $J$ satisfying $J=J^*=J^{-1}$) which defines the Krein space, is replaced by a (not necessarily self-adjoint) unitary, then we have the notion of an $S$-space which was introduced by Szafraniec. 
	In this paper, we consider $S$-spaces and study the structure of completely $U$-positive maps between the algebras of bounded linear operators. We first give a Stinespring-type representation for a completely $U$-positive map. On the other hand, we introduce Choi $U$-matrix of a linear map and  establish the equivalence of the Kraus $U$-decompositions and Choi $U$-matrices. Then we study properties of nilpotent completely $U$-positive maps. We develop the $U$-PPT criterion for separability of quantum $U$-states and discuss the entanglement breaking condition of quantum $U$-channels and explore $U$-PPT squared conjecture. Finally, we give concrete examples of  completely $U$-positive maps and examples of $3 \otimes 3$ quantum $U$-states which are $U$-entangled and $U$-separable.
\end{abstract}

	\maketitle

	\section{Introduction}
	The Gelfand-Naimark-Segal (GNS) construction 
	for a given state on a $C^*$-algebra provides us a representation of the $C^*$-algebra 
	on a Hilbert space and a generating vector. 
	A linear map $\tau$ from a $C^*$-algebra $\m B$ to a $C^*$-algebra 
	$\m C$ is said to be {\it completely positive (CP)} if $\sum_{i,j=1}^n c_j^{*}
	\tau(b_j^{*}b_i)c_i\geq 0$ whenever $b_1,b_2,\ldots,b_n\in\m B$; $c_1,c_2,\ldots,c_n\in\m C$ and $n\in \mathbb N$. Stinespring's 
	theorem (cf. \cite[Theorem 1] {St55}), which characterizes operator-valued completely positive maps, is a generalization of the GNS construction. Choi decomposition (cf. \cite{choi}) for completely positive maps is a pioneering work in Matrix Analysis.

	Dirac \cite{D42} and Pauli \cite{P43} were among the 
	pioneers to 
	explore the quantum field theory using Krein spaces, defined below. For our study, we require the following important definitions:
	
	\begin{definition} Assume $(\m K,\langle\cdot,\cdot\rangle)$ 
		to be a Hilbert space and $J$ to be a symmetry, that is, $J=J^*=J^{-1}$. 
		Define a map $[\cdot,\cdot]:\m K\times \m K\to  \mathbb{C}$ by
		\begin{eqnarray}
			[x,y]_J: =\langle Jx,y\rangle~\mbox{for all}~x,y\in \m K.
		\end{eqnarray}
		The tuple $(\m K,J)$ is called a {\it Krein space} (cf. \cite{B74}).
	\end{definition}

	\begin{definition}
		For each $V \in B(\m K),$ there exists an operator $V^{\natural}:=JV^{*}J \in B(\m K)$ such that
		\begin{align*}
			[Vx,y]_J &= \langle JVx,y\rangle= \langle x,V^{*}Jy\rangle=\langle x,J^{*}JV^{*}Jy\rangle\\&=\langle J x,JV^{*}Jy\rangle=\langle J x,V^{\natural}y\rangle=	[x,V^{\natural}y]_J.
		\end{align*} The operator $V^{\natural}$ is called the $J$-adjoint of $V.$ 
	\end{definition}
	
	In the definition of the Krein space, if we 
	replace the symmetry $J$ by a (not necessarily self-adjoint) unitary $U$,
	then we arrive at the following generalized notion due to Szafraniec \cite{Sz09}: 
	\begin{definition}
		Let $(\m H,\langle\cdot,\cdot\rangle)$ be a Hilbert space and let $U$ be a unitary on $\m H$, that is, $U^*=U^{-1}$. Then we can define a sesquilinear form by
		\begin{eqnarray}
			[x,y]_U: =\langle x,Uy\rangle~\mbox{for all}~x,y\in \m H.
		\end{eqnarray}
		In this case, we call {\rm $(\m H, U)$ as an $S$-space}. 
	\end{definition} The following definition is given by Phillipp, Szafraniec and Trunk, see \cite[Definition 3.1]{PST11}: \begin{definition}
		For each $V \in B(\m H),$ there exists an operator $V^{\#}:= UV^{*}U^{*} \in B(\m H)$ such that
		\begin{align*}
			[x,Vy]_U &= \langle x,UVy\rangle= \langle V^{*}U^{*}x,y\rangle=\langle U^{*}UV^{*}U^{*}x,y\rangle\\&=\langle UV^{*}U^{*}x,Uy\rangle=	[V^{\#}x,y]_U.
		\end{align*}The operator $V^{\#}$ is called the {\rm $U$-adjoint} of $V.$
	\end{definition} Phillipp, Szafraniec and Trunk \cite{PST11} 
	investigated invariant subspaces of self-adjoint operators
	in Krein spaces by using results obtained through a detailed analysis of S-spaces. Recently, in \cite{RR22}, Felipe-Sosa and Felipe introduced and analyzed the notions of state and quantum channel on spaces equipped with an indefinite metric in terms of a symmetry $J$. This study was further taken up by Heo, in \cite{He22}, where equivalence of Choi $J$-matrices and Kraus $J$-decompositions was obtained and applications to $J$-PPT criterion and $J$-PPT squared conjuncture were discussed. The notion of completely $U$-positive maps was studied by Dey and Trivedi in \cite{DT17,DT19}. Motivated by these inspiring works, in this paper, we develop structure theory of quantum $U$-channels and its applications to the entanglement breaking.
	
	The plan of the paper is as follows: In Section $\ref{sec1},$ we give Stinespring-type representation for a completely $U$-positive map. In Section $\ref{sec3},$ Choi $U$-matrix  is introduced and the equivalence of Kraus $U$-decompositions and Choi $U$-matrices is established. In Section $\ref{new},$ some properties of nilpotent $U$-CP maps are discussed. In Sections $\ref{sec4}$ and $\ref{sec5},$ we develop $U$-PPT criterion for separability of quantum $U$-states and discuss the entanglement breaking condition of quantum $U$-channels and explore $U$-PPT squared conjecture. Finally, in Section $\ref{sec6},$ we give concrete examples of completely $U$-positive maps and examples of $3 \otimes 3$ quantum $U$-states which are $U$-entangled and $U$-separable.
	
	\subsection{Background and notations}
	Let $(\m H, U)$ be an $S$-space. Then, $\m H^n$ is the direct sum of $n$-copies of the Hilbert space $\m H,$ and we denote by $(\m H^n,U^n)$ the $S$-space with the indefinite inner-product 
	\begin{align} [{\bf h},{\bf k}]_{U^n}=\left<{\bf h}, U^n{\bf k}\right>=\sum^n_{j=1} \left<h_j,Uk_j\right>=\sum^n_{j=1} [h_j,k_j]_U
	\end{align}	
	where $U^n=\mbox{diag}(U,U,\ldots,U)\in M_n(B(\m H))$ and ${\bf h}=(h_1,\ldots,h_n),~ {\bf k}=(k_1,\ldots,k_n)\in \m H^n.$
	
	\begin{definition}
		Let  $(\m H,U)$ be an $S$-space with the indefinite inner-product $[\cdot,\cdot]_{U}.$ We denote by $B(\m H)^{U+}$ the set of all $U$-positive linear operator $V$ on $\m H,$ that is, $$0\leq [Vh,h]_{U}:=\left< Vh, Uh\right>=\left< U^{*}Vh, h\right>, \mbox{ for all }~h\in \m H.$$  Hence $V$ is $U$-positive if  and only if $U^{*}V$ is positive with respect to the usual inner product $\langle\cdot,\cdot\rangle.$

	\end{definition}

	\begin{definition}\label{def1}
		Let $(\m H_i,U_i)~ (i=1,2)$ be an $S$-space with the indefinite inner-product $[\cdot,\cdot]_{U_i}.$ Let $\phi : B (\m H_1) \to  B (\m H_2) $ be a linear map. Then $\phi$ is called {\rm $(U_1,U_2)$-Hermitian} if $\phi(U_1V^*U^*_1)=U_2\phi(V^*)U^*_2$  for  $V\in B(\m H_1).$ We say that a $(U_1,U_2)$-Hermitian linear map $\phi$ is
		\begin{enumerate}
			\item {\it $(U_1,U_2)$-positive} if  $\phi(B(\m H_1)^{U+})\subset B(\m H_2)^{U+},$ that is, if $V \in (B(\m H_1))^{U+}$ (or $V$ is $U_{1}$-positive), then $\phi(V)$ is $U_2$-positive. In simple words, if $U_{1}^{*}V$ is positive  with respect to the usual inner product $\langle\cdot,\cdot\rangle_{\m H_1},$ then $U_{2}^{*}\phi (V)$ is positive  with respect to the usual inner product $\langle\cdot,\cdot\rangle_{\m H_2}.$
			\item {\it completely  $(U_1,U_2)$-positive} or
			$(U_1,U_2)$-{\it CP} if for each $l\in\mathbb N$ the $l$-fold amplification $\phi^{l}: I_l\otimes \phi: M_l(\mathbb C)\otimes B(\m H_1) \to M_l(\mathbb C) \otimes B(\m H_2) $ 
			defined by \begin{equation*}
				\phi^{l}([V_{ij}])=[\phi(V_{ij})], \:\:\:\:\text{for}\:\:\:\: [V_{ij}]\in M_l(B(\m H_1))
			\end{equation*}satisfies $$ \phi^{l}  (M_l(B(\m H_1))^{U+})\subset  M_l(B(\m H_2))^{U+},$$ that is, if $V=[V_{ij}]_{i,j} \in M_{l}( B(\m H_1))^{U+}$ (i.e., $V$ is $U_{1}^{l}$-positive), then $\phi^{l}(V)$ is $U_2^{l}$-positive. Here $M_l(B(\m H_i))^{U+}=B(\m H^l_i)^{U+}$ is the set of all $U^l_i$-positive linear operators on $S$-spaces $(\m H^l_i, U^l_i),$  and $U^l_{i}=\mbox{diag}(U,U,\ldots,U)$ $\in M_l(B(\m H_{i}))$ for $i=1,2.$
			\item {\it $U$-positive} (and {\it completely $U$-positive} ($U$-CP) )if $\m H_1=\m H_2=\m H$ and $U_1=U_2=U$ and it is ($U_1, U_2)$-positive (and $(U_1, U_2)$-CP, respectively).
		\end{enumerate}
	\end{definition}
		

		\section{Completely $U$-positive and completely $U$-co-positive maps}\label{sec1}
		Our main objective in this section is to obtain Stinespring-type theorem for completely $U$-positive maps.
		Let $(\m H_i, U_i)~(i=1,2)$ be an $S$-space with the indefinite inner product $[\cdot,\cdot]_{U_i}.$ Suppose $\phi : B (\m H_1) \to  B (\m H_2) $ is a linear map. Define a linear map $\psi$  from $B(\m H_1)$ to $B(\m H_2)$  by $\psi(X):=U_2\phi(U_1^{*}X)$ where $X \in B(\m H_1) .$ For any $l \in \mathbb{N}$ and $V=[V_{ij}]\in M_{l}( B (\m H_1)),$ we obtain
		\begin{align*}
			\psi^{l}(V)&=[\psi(V_{ij})]_{i,j}=[U_2\phi(U_1^{*}V_{ij})]_{i,j}= \begin{pmatrix}U_{2}\phi(U_{1}^{*}V_{11}) & \cdots  & U_{2} \phi(U_{1}^{*}V_{1l})\\ \vdots  & \ddots & \vdots \\U_{2}\phi(U_{1}^{*}V_{l1}) &\cdots & U_{2}\phi(U_{1}^{*}V_{ll}) \end{pmatrix}\\&=\begin{pmatrix}U_{2}  &  &  0\\  & \ddots & \\0 & & U_{2} \end{pmatrix} \begin{pmatrix}\phi(U_{1}^{*}V_{11}) & \cdots  &  \phi(U_{1}^{*}V_{ll})\\ \vdots  & \ddots & \vdots \\\phi(U_{1}^{*}V_{l1}) &\cdots & \phi(U_{1}^{*}V_{ll}) \end{pmatrix}=U_{2}^{l}\phi^{l}(U_1^{l^{*}}V).
		\end{align*}Similarly, we can easily show that   $\phi^{l}(V)=U^{l^*}_2 \psi(U_1^{l}V)$ where  $\phi(V_{ij})=U^*_2 \psi(U_1V_{ij}).$

		The following result is a generalization of \cite[Theorem 20]{RR22} and \cite[Proposition 2.2]{He22} in the setting of $S$-spaces:
		\begin{proposition} \label{prop1}
			Let $(\m H_i, U_i)~(i=1,2)$ be an $S$-space with the indefinite inner product $[\cdot,\cdot]_{U_i}.$ Suppose $\phi : B (\m H_1) \to  B (\m H_2) $ is a linear map, then $\phi$ is CP if and only if the corresponding linear map $\psi$  from $B(\m H_1)$ to $B(\m H_2)$ defined by $\psi(X):=U_2\phi(U_1^{*}X)$  is $(U_1,U_2)$-CP, where $X \in B(\m H_1).$
		\end{proposition}
		
		\begin{proof}
			Let $\phi$ be a linear map from $B(\m H_1)$ to $B(\m H_2).$  First assume that  $\phi$ is CP. We have to prove that  $\psi$ is $(U_1, U_2)$-CP. For this purpose, let $V=[V_{ij}] \in M_l(B(\m H_1))^{U+} ,$ that is, $U_1^{l^{*}}V\in M_l(B(\m H_1))$ is positive, that is, $$0 \leq [V {\bf{h}},{\bf{h}}]_{U_{1}^{l}}=\langle V {\bf{h}},U_{1}^{l}{\bf{h}}\rangle=\langle U_{1}^{l^{*}} V {\bf{h}},{\bf{h}}\rangle,$$ where ${\bf{h}}\in \mathcal{H_{1}}^{l}.$  Consider \begin{align*}
				[\psi^{l}(V){\bf{h'}}, {\bf{h'}}]_{U_{2}^{l}}&=\langle \psi^{l}(V){\bf{h'}}, U_{2}^{l} {\bf{h'}} \rangle=\langle U_{2}^{l^{*}}\psi^{l}(V){\bf{h'}},  {\bf{h'}} \rangle\\&=\langle U^{l}_2\phi^{l}(U_1^{l^{*}}V){\bf{h'}}, U_{2}^{l} {\bf{h'}} \rangle=\langle \phi^{l}(U_1^{l^{*}}V){\bf{h'}},  {\bf{h'}} \rangle \geq 0,
			\end{align*} where ${\bf{h'}}\in \mathcal{H_{2}}^{l}.$  Therefore $\langle U_{2}^{l^{*}}\psi^{l}(V){\bf{h'}},  {\bf{h'}} \rangle \geq 0,$ that is,  $U_{2}^{l^{*}}\psi^{l}(V)$ is positive.  This proves that $\psi(V)$ is $U_{2}$-positive.  Thus $\psi$ is $(U_1, U_2)$-CP.

			Conversely, suppose that $\psi$ is $(U_1,U_2)$-CP. Since $\psi(\cdot)=U_2\phi(U_1^{*}\cdot),$   we get $\phi(U_1^{*}\cdot)=U^*_2 \psi(\cdot).$ Therefore $\phi(\cdot)=U^*_2 \psi(U_1\cdot).$ Let $V=[V_{ij}] \in M_l(B(\m H_1))^{+} ,$ then we have to show that $\phi^{l}(V) =[\phi({V_{ij}})]\in M_l(B(\m H_2))^{+}.$ Now \begin{equation*}
				0 \leq	\langle V {\bf{h}},{\bf{h}} \rangle=\langle U_{1}^{l}V {\bf{h}}, U_{1}^{l}{\bf{h}} \rangle=[U_{1}^{l}V{\bf{h}},{\bf{h}}]_{U_{1}^{l}},
			\end{equation*} where ${\bf{h}}\in \mathcal{H_{1}}^{l},$ it means, $U_{1}^{l}V\in M_l(B(\m H_1))^{U+}.$ Therefore \begin{align*}
				\langle \phi^{l}(V){\bf{h'}},{\bf{h'}} \rangle&=\langle U^{l^{*}}_2 \psi(U_1^{l}V){\bf{h'}},{\bf{h'}} \rangle =\langle  \psi(U_1^{l}V){\bf{h'}}, U^{l}_2{\bf{h'}} \rangle\\&=[\psi(U_1^{l}V){\bf{h'}}, {\bf{h'}}]_{U_{2}^{l}}\geq0,
			\end{align*} where  ${\bf{h'}}\in \mathcal{H_{2}}^{l}$ and  the last inequality follows from the fact that $U_{1}^{l}V\in M_l(B(\m H_1))^{U+}$ and hence  $\psi$ is $(U_1,U_2)$-CP.
			
		\end{proof}

		\begin{theorem} \label{2.2} Let $(\m H_i, U_i)~(i=1,2)$ be an $S$-space. Assume that a linear map $\psi$  from $B(\m H_1)$ to $B(\m H_2)$ defined by $\psi(V):=U_2\phi(U_1^{*}V)$ for all $V\in B(\m H_1)$ is $(U_1,U_2)$-CP. Then there exist an $S$-space $(\m H,U),$ a $*$-representation $\pi$ of $B(\m H_1)$ on the Hilbert space $\m H$ and a bounded linear operator $R:\m H_2\to \m H$ such that $$\psi(V)=R^{\#}\pi(V)R$$ where $U=\pi(U_1),$ and $R^{\#}:=U_2 R^* U^{*}.$ Moreover, if $\psi(U_1)=U_2,$ then $R^* R=I_{\m H_2}.$		
		\end{theorem}
		\begin{proof}
			Suppose a linear map $\psi$  is $(U_1,U_2)$-CP. Then with the help of Proposition \ref{prop1}, we get that $\phi$ defined by $\phi(V)=U^*_2 \psi(U_1V)$  is CP. Then using Stinespring's theorem \cite[Theorem 1]{St55}, there exist a Hilbert space $\m H,$ a representation (a unital $*$-homomorphism ) $\pi$ of $B(\m H_1)$ on the Hilbert space $\m H$ and a bounded linear operator $R:\m H_2\to \m H,$ such that $\phi(V)=R^* \pi(V)R$ for every $V\in B(\m H_1).$
			
			Let $U=\pi(U_1)\in B(\m H),$ where $U$ is a fundamental unitary, that is, $U^*=U^{-1},$ so that $(\m H,U)$ becomes an $S$-space. Define $R^{\#}:=U_2 R^* U^{*},$ then
			\begin{align*}
				\psi(V)=U_2\phi(U_1^{*}V)=U_2R^* \pi(U_1^{*}V)R=U_2R^* U^{*}\pi(V)R=	R^{\#}\pi(V)R.
			\end{align*}
			
			Furthermore, if $\psi(U_1)=U_2,$ then \begin{align*}
				U_2=\psi(U_1)=U_2\phi(U_1^{*}U_1)=U_2R^* \pi(U_1^{*}U_1)R=U_{2}R^*R,
			\end{align*}  hence $R^* R=I_{\m H_2}.$
		\end{proof}
		
		\begin{theorem}
			Suppose $\phi : B (\m H_1) \to  B (\m H_2) $ is a linear map. If $\phi$ satisfies the following conditions for all $V \in B (\m H_1):$ \begin{align*}
				\phi(U_{1}^{*}V)=U_{2}^{*}\phi(V)\quad\mbox{and}\quad\phi(U_{1}V)=U_{2}\phi(V),
			\end{align*} then $\phi$ is a CP map if and only if $\phi$ is $(U_{1},U_{2})$-CP.
		\end{theorem}
		
		\begin{proof}
			First assume  $\phi$ to be a CP map. Let $V=[V_{ij}] \in M_l(B(\m H_1))^{U+}.$ Observe that 
			\begin{align*}
				\phi^{l}(U_1^{l^{*}}V)&=[\phi(U_1^{{*}}V_{ij})]_{i,j}= \begin{pmatrix}\phi(U_{1}^{*}V_{11}) & \cdots  &  \phi(U_{1}^{*}V_{1l})\\ \vdots  & \ddots & \vdots \\\phi(U_{1}^{*}V_{l1}) &\cdots & \phi(U_{1}^{*}V_{ll}) \end{pmatrix}\\&=\begin{pmatrix}U_{2}^{*}  &  &  0\\  & \ddots & \\0 & & U_{2}^{*} \end{pmatrix} \begin{pmatrix}\phi(V_{11}) & \cdots  &  \phi(V_{ll})\\ \vdots  & \ddots & \vdots \\\phi(V_{l1}) &\cdots & \phi(V_{ll}) \end{pmatrix}=U_{2}^{l^{*}}\phi^{l}(V).
			\end{align*}Similarly, we obtain $	\phi^{l}(U_1^{l}V)=U_{2}^{l}\phi^{l}(V).$ Now consider \begin{align*}
				[\phi^{l}(V){\bf{h'}}, {\bf{h'}}]_{U_{2}^{l}}&=\langle \phi^{l}(V){\bf{h'}}, U_{2}^{l} {\bf{h'}} \rangle=\langle U_{2}^{l^{*}}\phi^{l}(V){\bf{h'}},  {\bf{h'}} \rangle\\&=\langle \phi^{l}(U_1^{l^{*}}V){\bf{h'}},  {\bf{h'}} \rangle \geq 0,
			\end{align*} where ${\bf{h'}}\in \mathcal{H}_2^{l}.$  Therefore $\langle U_{2}^{l^{*}}\phi^{l}(V){\bf{h'}},  {\bf{h'}} \rangle \geq 0,$ that is,  $U_{2}^{l^{*}}\phi^{l}(V)$ is positive  with respect to the usual inner product $\langle\cdot,\cdot\rangle.$  This proves that $\phi(V)$ is $U_{2}$-positive.  Thus $\phi$ is $(U_1, U_2)$-CP.

			Conversely, suppose that $\phi$ is $(U_1,U_2)$-CP.  Let $V=[V_{ij}] \in M_l(B(\m H_1))^{+}.$ Then we have to show that $\phi^{l}(V) =[\phi({V_{ij}})]\in M_l(B(\m H_2))^{+}.$ Since \begin{equation*}
				0 \leq	\langle V {\bf{h}},{\bf{h}} \rangle=\langle U_{1}^{l}V {\bf{h}}, U_{1}^{l}{\bf{h}} \rangle=[U_{1}^{l}V{\bf{h}},{\bf{h}}]_{U_{1}^{l}},
			\end{equation*} where ${\bf{h}}\in \mathcal{H_{1}}^{l},$ it means $U_{1}^{l}V\in M_l(B(\m H_1))^{U+}.$ Then \begin{align*}
				\langle \phi^{l}(V){\bf{h'}},{\bf{h'}} \rangle&=\langle U^{l}_2 \phi(V){\bf{h'}}, U^{l}_2{\bf{h'}} \rangle =\langle  \phi(U_1^{l}V){\bf{h'}}, U^{l}_2{\bf{h'}} \rangle\\&=[\phi(U_1^{l}V){\bf{h'}}, {\bf{h'}}]_{U_{2}^{l}}\geq0,
			\end{align*} where  ${\bf{h'}}\in \mathcal{H_{2}}^{l}$ and  the last inequality follows from the fact that $U_{1}^{l}V\in M_l(B(\m H_1))^{U+}$ and  $\phi$ is $(U_1,U_2)$-CP.
		\end{proof}
		
		\begin{remark} In particular, if $\m H_1=\m H_2 =\m H$ and $U_{1}=U_{2}=U,$ and if  a linear map $\phi : B (\m H) \to  B (\m H) $ satisfies $	\phi(U^{*}V)=U^{*}\phi(V)\mbox{ and } \phi(UV)=U\phi(V)$ for all $V \in B (\m H_1),$ then $\phi$ is CP if and only if $\phi$ is $U$-CP.
		\end{remark}
		
		\begin{definition}
			Let $(\m H_i, U_i)~(i=1,2)$ be an $S$-space. Assume that $\psi$ is a linear map  from $B(\m H_1)$ to $B(\m H_2).$ Then
			\begin{enumerate}
				\item for each $ l \in \mathbb{N},$ $\psi$ is {\rm $l$-$(U_1,U_2)$-co-positive} if $\tau_{l} \otimes \psi: M_l(\mathbb C)\otimes B(\m H_1) \to M_l(\mathbb C) \otimes B(\m H_2) $ is $(I_{l}\otimes U_1, I_{l}\otimes U_2)$-positive where $\tau_{l}$ is the transpose map on $M_l(\mathbb C).$
				\item $\psi$ is {\rm completely $(U_1,U_2)$-co-positive} if it is $l$-$(U_1,U_2)$-co-positive for each  $ l \in \mathbb{N}.$
				\item $\psi$ is {\rm $(U_1,U_2)$-positive partial transpose ($(U_1,U_2)$-PPT)} if it is $(U_1,U_2)$-CP and completely  $(U_1,U_2)$-co-positive.
				\item In particular, if $\m H_1=\m H_2 =\m H$ and $U_{1}=U_{2}=U,$ then we simply call it completely $U$-co-positive (and $U$-positive partial transpose ($U$-PPT)) if it is completely  $(U_1, U_2)$-co-positive (and $(U_1, U_2)$-positive partial transpose, respectively).
			\end{enumerate}
		\end{definition}

		\begin{proposition} \label{prop2}
			Let $(\m H_i, U_i)~(i=1,2)$ be an $S$-space. Suppose $\phi : B (\m H_1) \to  B (\m H_2) $ is a linear map, then $\phi$ is completely co-positive if and only if the corresponding linear map $\psi$  from $B(\m H_1)$ to $B(\m H_2)$ defined by $\psi(X):=U_2\phi(U_1^{*}X)$  is completely $(U_1,U_2)$-co-positive, where $X \in B(\m H_1).$
		\end{proposition}
		
		\begin{proof}
			Let $V=[V_{ij}]\in M_l(\mathbb C)\otimes B(\m H_1)$ be such that $(I_{l}\otimes U_{1}^{*})V \geq 0 .$ Then \begin{align*}
				(\tau_{l}	\otimes\psi)(V)&=\begin{pmatrix}\psi(V_{11}) & \cdots  & \psi(V_{l1})\\ \vdots  & \ddots & \vdots \\\psi(V_{1l}) &\cdots & \psi(V_{ll}) \end{pmatrix}=\begin{pmatrix}U_2\phi(U_1^{*}V_{11}) & \cdots  & U_2\phi(U_1^{*}V_{l1})\\ \vdots  & \ddots & \vdots \\U_2\phi(U_1^{*}V_{1l}) &\cdots & U_2\phi(U_1^{*}V_{ll}) \end{pmatrix}\\&=\begin{pmatrix}U_{2}  &  &  0\\  & \ddots & \\0 & & U_{2} \end{pmatrix} \begin{pmatrix}\phi(U_1^{*}V_{11}) & \cdots  & \phi(U_1^{*}V_{l1})\\ \vdots  & \ddots & \vdots \\\phi(U_1^{*}V_{1l}) &\cdots & \phi(U_1^{*}V_{ll}) \end{pmatrix}\\&=(I_{l}\otimes U_{2})(\tau_{l} \otimes \phi)(I_{l}\otimes U_{1}^{*})V.
			\end{align*} Hence $(I_{l}\otimes U_{2}^{*})(\tau_{l}	\otimes\psi)(V)$ is positive as $\phi$ is completely co-positive map.
			
			Conversely,  for any  $V=[V_{ij}] \in M_l(B(\m H_1) ,$ we have \begin{equation*}
				0 \leq	\langle V {\bf{h}},{\bf{h}} \rangle=\langle U_{1}^{l}V {\bf{h}}, U_{1}^{l}{\bf{h}} \rangle=[U_{1}^{l}V{\bf{h}},{\bf{h}}]_{U_{1}^{l}},
			\end{equation*} where ${\bf{h}}\in \mathcal{H_{1}}^{l},$ it means $U_{1}^{l}V\in M_l(B(\m H_1))^{U+}.$  We obtain 
			
			\begin{align*}
				(\tau_{l} \otimes \phi)(V)&=\begin{pmatrix}\phi(V_{11}) & \cdots  & \phi(V_{l1})\\ \vdots  & \ddots & \vdots \\\phi(V_{1l}) &\cdots & \phi(V_{ll}) \end{pmatrix}=\begin{pmatrix}U_2^{*}\psi(U_1V_{11}) & \cdots  & U_2^{*}\psi(U_1V_{l1})\\ \vdots  & \ddots & \vdots \\U_2^{*}\psi(U_1V_{1l}) &\cdots & U_2^{*}\psi(U_1V_{ll}) \end{pmatrix}\\&=\begin{pmatrix}U_{2}^{*}  &  &  0\\  & \ddots & \\0 & & U_{2}^{*} \end{pmatrix} \begin{pmatrix}\psi(U_1V_{11}) & \cdots  & \psi(U_1V_{l1})\\ \vdots  & \ddots & \vdots \\\psi(U_1V_{1l}) &\cdots & \psi(U_1V_{ll}) \end{pmatrix}\\&=U_{2}^{l^{*}}(\tau_{l} \otimes \psi)(U_1^{l}V).
			\end{align*} Therefore $	(\tau_{l} \otimes \phi)(V)=U_{2}^{l^{*}}(\tau_{l} \otimes \psi)(U_1^{l}V).$ Since $U_{1}^{l}V\in M_l(B(\m H_1))^{U+}$ and $\psi$ is completely $(U_1,U_2)$-co-positive, $\phi$ is co-positive.
			
		\end{proof}

		\section{Kraus $U$-decomposition and Choi $U$-matrix}\label{sec3}In this section, we derive Kraus $U$-decomposition and Choi $U$-matrix and establish their relation with the completely $U$-positive maps.
		Let $M_{m}(\mathbb{C})$ denote the set of all $ m \times m$-complex matrices. Kraus proved that $\phi : M_{m}(\mathbb{C}) \to M_{n}(\mathbb{C})$ is a CP map if and only if \begin{equation}\label{kraus1}
			\phi(V)=\sum_{i=1}^{l}R_{i}^{*}VR_{i}, 
		\end{equation} where $V=[V_{ij}]_{i,j} \in M_{m}(\mathbb{C})$ and for each $i,$ $R_{i} \in M_{m,n}(\mathbb{C}).$ The expression in above equation is called a Kraus decomposition.

		Denote $M_{A}:=M_{m}(\mathbb{C})$ and $M_{B}:=M_{n}(\mathbb{C}).$  Let  $U_{A}$ and $U_{B}$ be the fundamental unitaries in $M_{A}$ and $ M_{B},$ respectively. Define a linear map $
		\psi: M_{A} \to M_{B}$ by  \begin{equation}\label{kraus}\psi(V):=\sum_{i=1}^{l}R_{i}^{\#_{A,B}}VR_{i},
		\end{equation} where $R_{i}^{\#_{A,B}}= U_B R^*_i U_{A}^{*}.$	Then $\psi$ is $(U_{A},U_{B})$-CP.	Indeed, for any $k\in \mathbb{N},$ take a $U_{A}^{k^{*}}$-positive matrix $V=[V_{ij}]\in M_{k}(M_{A})^{U+}.$  Since $V=[V_{ij}]\in M_{k}(M_{A})^{U+},$ $U_{A}^{k^{*}}V \in M_{k}(M_{A})^{+},$ that is, \begin{align*}
			U_{A}^{k^{*}}V&=\begin{pmatrix}U_{A}^{*}  &  &  0\\  & \ddots & \\0 & & U_{A}^{*} \end{pmatrix} \begin{pmatrix}V_{11} & \cdots  &  V_{1k}\\ \vdots  & \ddots & \vdots \\V_{k1} &\cdots & V_{kk} \end{pmatrix}\\& = \begin{pmatrix}U_{A}^{*}V_{11} & \cdots  &  U_{A}^{*}V_{1k}\\ \vdots  & \ddots & \vdots \\U_{A}^{*}V_{k1} &\cdots & U_{A}^{*}V_{kk} \end{pmatrix} \in M_{k}(M_{A})^{+}.
		\end{align*}
		
		Consider \begin{align*}
			\psi^{k}(V)&=\psi^{k}\begin{pmatrix}V_{11} & \cdots  &  V_{1k}\\ \vdots  & \ddots & \vdots \\V_{k1} &\cdots & V_{kk} \end{pmatrix}=\begin{pmatrix}\psi(V_{11}) & \cdots  &  \psi(V_{1k})\\ \vdots  & \ddots & \vdots \\ \psi(V_{k1}) &\cdots & \psi(V_{kk} )\end{pmatrix}\\&=\sum_{i=1}^{l}\begin{pmatrix} R_{i}^{\#_{A,B}}V_{11}R_{i} & \cdots  &  R_{i}^{\#_{A,B}}V_{1k}R_{i}\\ \vdots  & \ddots & \vdots \\ R_{i}^{\#_{A,B}}V_{k1}R_{i} &\cdots & R_{i}^{\#_{A,B}}V_{kk}R_{i}\end{pmatrix}\\&=\sum_{i=1}^{l}\begin{pmatrix} U_B R^*_i U_{A}^{*}V_{11}R_{i} & \cdots  &  U_B R^*_i U_{A}^{*}V_{1k}R_{i}\\ \vdots  & \ddots & \vdots \\ U_B R^*_i U_{A}^{*}V_{k1}R_{i} &\cdots & U_B R^*_i U_{A}^{*}V_{kk}R_{i}\end{pmatrix}\\&=\sum_{i=1}^{l}\begin{pmatrix}U_{B}  &  &  0\\  & \ddots & \\0 & & U_{B} \end{pmatrix}\begin{pmatrix}R_{i}^{*} &  &  0\\  & \ddots & \\0 & & R_{i}^{*} \end{pmatrix}	U_{A}^{k^{*}}V \begin{pmatrix}R_{i} &  &  0\\  & \ddots & \\0 & & R_{i} \end{pmatrix}\\&=U_{B}^{k}\sum_{i=1}^{l} \begin{pmatrix}R_{i}^{*} &  &  0\\  & \ddots & \\0 & & R_{i}^{*} \end{pmatrix}	U_{A}^{k^{*}}V \begin{pmatrix}R_{i} &  &  0\\  & \ddots & \\0 & & R_{i} \end{pmatrix},
		\end{align*} and since $U_{A}^{k^{*}}V \in M_{k}(M_{A})^{+},$  by using the Kraus decomposition $$\sum_{i=1}^{l}R_{i}^{*^{k}}U_{A}^{k^{*}}VR_{i}^{k} \in M_{k}(M_{A})^{+},$$ we obtain $U_{B}^{k^{*}}\psi^{k}(V) \geq 0.$ Hence $\psi^{k}(V)$ is a $U_{B}$-positive matrix, that is,  $\psi$ is $(U_{A},U_{B})$-CP map.

		\begin{theorem}
			Let  $U_{A}$ and $U_{B}$ be the fundamental unitaries in $M_{A}$ and $ M_{B},$ respectively.  A linear map $
			\psi: M_{A} \to M_{B}$ is a $(U_{A},U_{B})$-CP map if and only if it has a decomposition of the form (\ref{kraus}).
		\end{theorem}

		\begin{proof}
			Assume that $\psi$ is a $(U_{A},U_{B})$-CP map. Since a linear map $\phi: M_{A}\to M_{B}$ defined by
			$\phi(V)=U^*_B \psi(U_AV)$  is CP, $\phi$ has a Kraus  decomposition, that is,\begin{equation*}
				\phi(V)=\sum_{i=1}^{l}R_{i}^{*}VR_{i}, 
			\end{equation*} where $V \in M_{m}(\mathbb{C})$ and for each $i,$ $R_{i} \in M_{m,n}(\mathbb{C}).$ Thus we have \begin{align*}
				\psi(V)=U_B\phi(U_A^{*}V)=U_B\sum_{i=1}^{l}R_{i}^{*}U_A^{*}VR_{i}=\sum_{i=1}^{l}U_BR_{i}^{*}U_A^{*}VR_{i}=\sum_{i=1}^{l}R_{i}^{\#}VR_{i}.
			\end{align*} Therefore $\psi$ is a $(U_{A},U_{B})$-CP map if and only if $\psi$ has the expression $\psi(V)=\sum_{i=1}^{l}R_{i}^{\#}VR_{i},$ we call $\psi$ has a {\it Kraus $U$-decomposition} in this case.
		\end{proof} Suppose	 $\{e_{ij}\:\:|\:\: 1 \leq i,j \leq m\}$ are the matrix units of $M_{m}(\mathbb{C}).$ We observe that $D=[U_{A}e_{ij}]_{1\leq i,j \leq m}$ is $I_{m}\otimes U_{A}$-positive. Indeed,\begin{align*}
			(I_{m}\otimes U_{A}^{*})D &=\begin{pmatrix}U_{A}^{*}  &  &  0\\  & \ddots & \\0 & & U_{A}^{*} \end{pmatrix} \begin{pmatrix}U_{A}e_{11} & \cdots  &  U_{A}e_{1m}\\ \vdots  & \ddots & \vdots \\U_{A}e_{m1} &\cdots & U_{A}e_{mm} \end{pmatrix}\\&=\begin{pmatrix}e_{11} & \cdots  &  e_{1m}\\ \vdots  & \ddots & \vdots \\e_{m1} &\cdots & e_{mm} \end{pmatrix}\in M_{m^{2}}^{+}(\mathbb{C}).
		\end{align*}It implies from the above proposition that $[\psi (U_{A}e_{ij})]_{1\leq i,j \leq m}$ is $I_{m}\otimes U_{B}$-positive.
		\begin{theorem}
			Let   $\psi: M_{A} \to M_{B}$  be a  linear map.	Then $\psi$ is $(U_{A},U_{B})$-CP if and only if $[U_{B}^{*}\psi (U_{A}e_{ij})]_{1\leq i,j \leq m}$ is positive.
		\end{theorem}
		\begin{proof}
			The proof directly  follows from \cite[Theorem 2]{choi}.
		\end{proof}
		
		Let  $\phi : M_{m}(\mathbb{C}) \to M_{n}(\mathbb{C})$  be  a linear  map. Choi \cite{choi} defined  $C_{\phi}=\sum_{i,j=1}^{m} e_{ij}\otimes \phi(e_{ij}),$ called  the {\it Choi matrix}, and proved that it is positive if and only if $\phi$ is a CP map.
		\begin{definition}
			Let  $\psi : M_{m}(\mathbb{C}) \to M_{n}(\mathbb{C})$  be  a linear  map. We define $C_{\psi}^{U}:=\sum_{i,j=1}^{m} e_{ij}\otimes \psi(U_Ae_{ij}).$ The matrix $C_{\psi}^{U}$ is called the {\rm Choi $U$-matrix}.
		\end{definition} 
		\begin{theorem}
			Let  $U_{A}$ and $U_{B}$ be the fundamental unitaries in $M_{A}$ and $ M_{B},$ respectively, where  $M_{A}=M_{m}(\mathbb{C})$ and $M_{B}=M_{n}(\mathbb{C}).$  Then a linear map  $\psi: M_{A} \to M_{B}$ is a  $(U_{A},U_{B})$-CP map if and only if $C_{\psi}^{U}$ is $I_{A} \otimes U_{B}$-positive in $M_{A}\otimes M_{B}.$
		\end{theorem}\begin{proof}
			Let $\phi : M_{A} \to M_{B}$ be the linear map defined by  $\phi(V):=U^*_B \psi(U_AV)$  where $V\in M_{A}.$ Then by Proposition \ref{prop1}, $\phi$ is CP if and only if $\psi$ is a $(U_{A},U_{B})$-CP  map. It is known from \cite{choi}  that $\phi$ is CP  if and only if $C_{\phi}$ is positive semi-definite.
			Since, for any ${\bf{h}},{\bf{h'}} \in \mathbb{C}^{mn},$ we have \begin{align*}
				[C_{\psi}^{U}{\bf{h}},{\bf{h'}}]_{U_{B}^{m}}&=\langle C_{\psi}^{U}{\bf{h}},{U_{B}^{m}}{\bf{h'}} \rangle=\langle {U_{B}^{m^{*}}} C_{\psi}^{U}{\bf{h}},{\bf{h'}} \rangle\\&=\langle \begin{pmatrix}U_{B}^{*}\psi(U_{A}e_{11}) & \cdots  &  U_{B}^{*}\psi(U_{A}e_{1m})\\ \vdots  & \ddots & \vdots \\U_{B}^{*}\psi(U_{A}e_{m1}) &\cdots & U_{B}^{*}\psi(U_{A}e_{mm}) \end{pmatrix}{\bf{h}},{\bf{h'}} \rangle\\& =\langle \begin{pmatrix}\phi(e_{11}) & \cdots  &  \phi(e_{1m})\\ \vdots  & \ddots & \vdots \\\phi(e_{1m}) &\cdots & \phi(e_{mm}) \end{pmatrix}{\bf{h}},{\bf{h'}} \rangle\\&=\langle C_{\phi} {\bf{h}},{\bf{h'}} \rangle,
			\end{align*} that is, $C_{\phi}$ is positive if and only if $C_{\psi}^{U}$ is $I_{A} \otimes U_{B}$-positive in $M_{A}\otimes M_{B},$ which completes the proof. 
		\end{proof}
		\section{Nilpotent  $U$-CP maps}\label{new}
		Nilpotent CP maps were studied by Bhat and Mallick in \cite{bhat14}. Let ${\m H}$ be a finite dimensional Hilbert space and  $\phi : B (\m H) \to  B (\m H) $ be a CP map. Suppose $\phi$ is a nilpotent map of order $p,$ that is, $\phi^{p}=0$ and $\phi^{p-1}\neq 0.$ Define ${\m H}_1:=\ker(\phi(U))$ and ${\m H}_k:=\ker(\phi^{k}(U))\ominus \ker(\phi^{k-1}(U)),$ where $ 2 \leq k\leq p.$ Then $\cap_{k=1}^{p}{\m H}_{k}=\emptyset$ and $\m H= {\m H}_1 \oplus {\m H}_2\oplus \cdots \oplus {\m H}_p.$	Let $b_i:=\dim({\m H}_i)$ for $1 \leq i \leq p.$ Then $(b_1,b_2, \ldots, b_p)$ is called the {\rm CP nilpotent type of $\phi.$} In this section, we introduce  $U$-CP nilpotent type of $U$-CP maps.
		
		\begin{proposition}Let ${\m H}$ be a finite dimensional Hilbert space and $(\m H, U)$ be an $S$-space with the indefinite inner product $[\cdot,\cdot]_{U}.$ Suppose $\phi : B (\m H) \to  B (\m H) $ is a CP map, then  the corresponding linear map $\psi$  from $B(\m H)$ to $B(\m H)$ defined by $\psi(X):=U\phi(U^{*}X)$  is $U$-CP, with the  Kraus $U$-decomposition $\psi(X)=\sum_{i=1}^{l}R_{i}^{\#}XR_{i},$ where $X \in B(\m H)$ and $ R_{i}^{\#}=UR_{i}^{*}U^{*}$ for each $1\leq i \leq l.$ Then 
			\begin{enumerate}
				\item $\ker(\psi(U))=\cap_{i=1}^{l}\ker(UR_i),$
				\item For $U$-positive $X,$ $\psi(X)=0$ if and only if $\text{ran}(X)\subseteq\cap_{i=1}^{l}\ker(R_i^{*}U^{*}),$
				\item $\{h \in{\m H}\:\:|\:\: \psi(|Uh \rangle\langle h|  )=0\}=\cap_{i=1}^{l}\ker(R_i^{*}U^{*}),$
				\item $\text{ran}(\psi(U))=\overline{span}\{UR_{i}^{*}h\:\:|\:\: h \in {\m H},\:\: 1 \leq i \leq l\}.$
			\end{enumerate}
		\end{proposition}
		\begin{proof}$(1)$
			Consider \begin{align*}
				\ker(\psi(U))&=\{h\in {\m H} \:\:|\:\ \psi(U)h=0\}\\&=\{h\in {\m H} \:\:|\:\ \sum_{i=1}^{l}R_{i}^{\#}UR_{i}h=0\}\\&=\{h\in {\m H} \:\:|\:\ \sum_{i=1}^{l}[R_{i}^{\#}UR_{i}h,h]_{U}=0\}\\&=\{h\in {\m H} \:\:|\:\ \sum_{i=1}^{l}[UR_{i}h,R_{i}h]_{U}=0\}\\&=\{h\in {\m H} \:\:|\:\ \sum_{i=1}^{l}\langle UR_{i}h,UR_{i}h\rangle=0\}\\&=\{h\in {\m H} \:\:|\:\ \sum_{i=1}^{l} \|UR_{i}h\|^{2}=0\}\\&=\{h\in {\m H} \:\:|\:\ UR_{i}h=0,\:\:\text{for each}\:\: 1\leq i\leq l\}\\&=\bigcap_{i=1}^{l}\ker(UR_i).
			\end{align*}
			
			$(2)$ Suppose $\psi(X)=U\phi(U^{*}X)=0$ where $X$ is $U$-positive. It follows that $\phi(U^{*}X)=0,$   and since $\phi$ is a CP map, using the Kraus decomposition, we obtain	$\sum_{i=1}^{l}R_{i}^{*}U^{*}XR_{i}=0.$ As $X$ is $U$-positive ($U^{*}X$ is positive), we get $R_{i}^{*}U^{*}XR_{i}=0$ for each $i.$ Note that $R_{i}^{*}(U^{*}X)^{\frac{1}{2}}=0.$ It implies that $R_{i}^{*}U^{*}X=0.$ Let $h_{1}\in {\rm ran}(X),$ then there exists $h_{2}\in \m H$ such that $X(h_{2})=h_1.$ Now by applying $R_{i}^{*}U^{*}$ on both the sides, we get $R_{i}^{*}U^{*}{h}_1=0$ for each $i.$ Hence ${\rm ran}(X)\subseteq\cap_{i=1}^{l}\ker(R_i^{*}U^{*}).$
			
			Conversely, let ${\rm ran}(X)\subseteq\cap_{i=1}^{l}\ker(R_i^{*}U^{*}),$ then  $\psi(X)=\sum_{i=1}^{l}R_{i}^{\#}XR_{i}=\sum_{i=1}^{l}UR_{i}^{*}U^{*}XR_{i}=0.$

			$(3)$ One can easily see that $|Uh \rangle\langle h| $ is $U$-positive. Indeed, $U^{*}|Uh \rangle\langle h|=|h \rangle\langle h| \geq 0.$ Also, we have $\psi(|Uh \rangle\langle h|)=0,$ and  ${\rm ran}(|Uh \rangle\langle h|)=\mathbb{C}h,$ therefore it directly follows from $(2)$ that $\{h \in {\m H} \:\:|\:\: \psi(|Uh \rangle\langle h|  )=0\}=\cap_{i=1}^{l}\ker(R_i^{*}U^{*}).$

			$(4)$ Let $ h_1 \in {\rm ran}(\psi(U))={\rm ran}(\sum_{i=1}^{l}R_{i}^{\#}UR_{i})={\rm ran}(\sum_{i=1}^{l}UR_{i}^{*}R_{i}).$ Then  $\sum_{i=1}^{l}UR_{i}^{*}R_{i}h_2=h_1$ for some $h_2 \in {\m H}.$ Therefore $ h_1 \in \overline{span}\{UR_{i}^{*}h \:\:|\:\: h \in {\m H},\:\: 1 \leq i \leq l\}.$ Hence ${\rm ran}(\psi(U)) \subseteq \overline{span}\{UR_{i}^{*}h \:\:|\:\: h \in {\m H},\:\: 1 \leq i \leq l\}.$
			
			Conversely, let $ h \in \overline{span}\{UR_{i}^{*}h \:\:|\:\: h \in {\m H},\:\: 1 \leq i \leq l\}.$ Then $ h =\sum_{i=1}^{l}\alpha_{i}UR_{i}^{*}h_i$ where $\alpha_{i} \in \mathbb{C},$ $h_i \in {\m H}.$ We have to show that $h \in {\rm ran}(\psi(U))={\rm ran}(U\sum_{i=1}^{l}R_{i}^{*}R_{i}).$ It is equivalent to show that $ h \in \ker(\sum_{i=1}^{l}R_{i}^{*}R_{i}U^{*})^{\perp},$ that is, $\langle h,h'\rangle_{\m H}=0$ for all $h'\in \ker(\sum_{i=1}^{l}R_{i}^{*}R_{i}U^{*}).$
			
			Consider $h'\in$   $\ker(\sum_{i=1}^{l}R_{i}^{*}R_{i}U^{*}),$ then we have \[0=\sum_{i=1}^{l}[R_{i}^{*}R_{i}U^{*}h',h']_{U^{*}}=\sum_{i=1}^{l}\langle R_{i}^{*}R_{i}U^{*}  h',U^{*}h'\rangle.\] It follows that $R_{i}U^{*}h'=0$ for each $i.$ Observe that \begin{align*}
				\langle h,h'\rangle&=\sum_{i=1}^{l}\alpha_{i}\langle UR_{i}^{*}h_i,h'\rangle=\sum_{i=1}^{l}\alpha_{i}\langle h_i,R_iU^{*}h'\rangle=0,
			\end{align*} which proves that $\text{ran}(\psi(U))=\overline{span}\{UR_{i}^{*}h\:\:|\:\: h \in {\m H},\:\: 1 \leq i \leq l\}.$
		\end{proof}
		\begin{proposition}
			Let $(\m H, U)$ be an $S$-space with the indefinite inner product $[\cdot,\cdot]_{U}.$ Suppose $\phi : B (\m H) \to  B (\m H) $ is a CP map, then  the corresponding linear map $\psi$  from $B(\m H)$ to $B(\m H)$ defined by $\psi(X):=U\phi(U^{*}X)$  is $U$-CP, with the  Kraus $U$-decomposition $\psi(X)=\sum_{i=1}^{l}R_{i}^{\#}XR_{i},$ where $X \in B(\m H)$ and $ R_{i}^{\#}=UR_{i}^{*}U^{*}$ for each $1\leq i \leq l.$ Then the followings are equivalent: \begin{enumerate}
				\item $\psi^{p}(X)=0$ for all $X \in  B(\m H);$ \item $R_{i_1}R_{i_2}\cdots R_{i_p}=0$ for all $i_1,i_2,\ldots,i_p.$
			\end{enumerate}
		\end{proposition}
		\begin{proof} $(1) \implies (2):$ Let us assume
			for each $X \in  B(\m H),$ we have  \begin{align*}0=
				\psi^{p}(X)=\sum_{i_1,i_2,\ldots,i_p=1}^{l}R_{i_{p},\ldots,i_1}^{\#}XR_{i_1}R_{i_2}\cdots R_{i_p},
			\end{align*} where $R_{i_{p},\ldots,i_1}^{\#}=UR_{i_p}^{*}R_{i_{p-1}}^{*}\cdots R_{i_1}^{*}U^{*}.$ Therefore  $$0=\psi^{p}(I)=\sum_{i_1,i_2,\ldots,i_p=1}^{l}R_{i_{p},\ldots,i_1}^{\#}R_{i_1}R_{i_2}\cdots R_{i_p} .$$
			Now observe that
			\begin{align*}&
				\{h\in {\m H} \:\:|\:\ \sum_{i_1,i_2,\ldots,i_p=1}^{l}R_{i_{p},\ldots,i_1}^{\#}R_{i_1}R_{i_2}\cdots R_{i_p}h=0\}\\=&\{h\in {\m H} \:\:|\:\ \sum_{i_1,i_2,\ldots,i_p=1}^{l}[R_{i_{p},\ldots,i_1}^{\#}R_{i_1}R_{i_2}\cdots R_{i_p}h,h]_{U}=0\}\\=&\{h\in {\m H} \:\:|\:\ \sum_{i_1,i_2,\ldots,i_p=1}^{l}[R_{i_1}R_{i_2}\cdots R_{i_p}h,R_{i_1}R_{i_2}\cdots R_{i_p}h]_{U}=0\},
			\end{align*} which concludes the desired equality $(2).$
			
			$(2) \implies (1):$ Trivial.
		\end{proof}
		
		Suppose $\psi$ is a $U$-CP map from $B(\m H)$ to $B(\m H)$ defined by $\psi(X)=U\phi(U^{*}X).$ Let $\psi$ be a nilpotent map of order $p.$  Define ${\m K}_1:=\ker(\psi(U))$ and ${\m K}_k:=\ker(\psi^{k}(U))\ominus \ker(\psi^{k-1}(U)),$ where $ 2 \leq k\leq p.$ Then $\cap_{k=1}^{p}{\m K}_{k}=\emptyset$ and $\m H= {\m K}_1 \oplus {\m K}_2\oplus \cdots \oplus {\m K}_p.$
		\begin{definition}
			Let $c_i:=\dim({\m K}_i)$ for $1 \leq i \leq p.$ Then $(c_1,c_2, \ldots, c_p)$ is called the { \rm $U$-CP nilpotent type of $\psi.$}
		\end{definition}

		\section{Quantum U-channels and quantum $U$-states}\label{sec4}

		The $U$-states and the quantum $U$-channel, which are the $S$-space versions of the states and quantum channel, respectively, are introduced in this section. Together, we introduce $U$-separable and $U$-entangled states and present the $U$-PPT criterion for $U$-separability of $U$-states.
		\begin{definition}
			Let $\phi : M_{A} \to M_{B}$ be a linear map and  $U_{A}$ and $U_{B}$ be the fundamental unitaries in $M_{A}$ and $ M_{B},$ respectively. Then
			\begin{enumerate}
				\item $\phi $ is a {\rm quantum channel} if it is CP and trace preserving, that is, $\text{Tr}(\phi(V))=\text{Tr}(V)$ where $V \in  M_{A}.$
				\item a linear map $\psi$  from $B(\m H_1)$ to $B(\m H_2)$ defined by $\psi(V):=U_2\phi(U_1^{*}V)$  is a {\rm quantum $(U_{A},U_{B})$-channel} if it is $(U_{A},U_{B})$-CP and trace preserving.
			\end{enumerate}
		\end{definition}
		\begin{remark}
			It is well known that $\phi$ is a quantum channel if and only if there exist $ m \times n$-matrices $R_1,\ldots,R_{l}$ such that 
			\begin{align*}
				\phi(V)=\sum_{i=1}^{l}R_{i}^{*}VR_{i}\:\:\:\:\mbox{and}\:\:\:\: \sum_{i=1}^{l}R_{i}R_{i}^{*}=I
			\end{align*} where $V \in M_{A}.$ Indeed, if $\phi$ is a quantum channel, then it is a CP map and trace preserving. Therefore by Kraus decomposition (\ref{kraus1}), there exist $ m \times n$-matrices $R_1,\ldots,R_{l}$ such that $	\phi(V)=\sum_{i=1}^{l}R_{i}^{*}VR_{i},$ and if $\phi$ is a trace preserving map, then $\phi^{*}(V)=\sum_{i=1}^{l}R_{i}VR_{i}^{*}$ is unital ($ \text{Tr}(X)=\langle I_{X},X \rangle= \text{Tr}(\phi(X))=\langle I_{X},\phi(X) \rangle=\langle \phi^{*}(I_{X}),X \rangle$) which implies $\sum_{i=1}^{l}R_{i}R_{i}^{*}=I.$
		\end{remark}
		Similarly, if $\psi$ is a quantum $(U_{A}, U_{B})$-channel, then by Kraus $U$-decomposition (\ref{kraus}) we have  $\psi(V)=\sum_{i=1}^{l}R_{i}^{\#_{A,B}}VR_{i},$ where $R_{i}^{\#_{A,B}}= U_B R^*_i U_{A}^{*}.$ Since $\psi$ is trace preserving, it means  $\psi^{*}$ is unital and we obtain $I_{B}=\psi^{*}(I_{A})= \sum_{i=1}^{l}R_{i}R_{i}^{\#_{A,B}}.$ Moreover, \begin{align*}
			\sum_{i}R_{i}U_{B}^{*}R_{i}^{\#_{A,B}}=R_{i}U_{B}^{*}U_B R^*_i U_{A}^{*}=U_{A}^{*}.
		\end{align*}
		A {\it quantum state} $\rho \in M_{n}(\mathbb{C})$ is a positive semi-definite matrix with $\text{Tr}(\rho)=1.$\begin{definition}
			Let $U$ be a fundamental unitary in $M_{n}(\mathbb{C}),$ then a matrix $\rho \in M_{n}(\mathbb{C})$ is called a {\rm quantum $U$-state} if the following conditions hold:
			\begin{enumerate}
				\item $\rho$ is $U$-positive, that is, $U^{*}\rho$ is positive and
				\item $\text{Tr}(U^{*}\rho)=1.$
			\end{enumerate}
		\end{definition}
		
		\begin{example}
			Let $U$ be a fundamental unitary in $M_{l}(\mathbb{C}),$ where $ l \in \mathbb{N}.$ Define $\rho\in M_{l}(\mathbb{C})$ as $\rho = |Ue\rangle \langle e|$ where $e\in\mathbb{C}^l$ with $ \|e\|=1.$ Then
			\begin{align*}
				&U^{*}\rho=U^{*}|Ue\rangle \langle e|=|U^*Ue\rangle \langle e|=|e\rangle \langle e|.
			\end{align*} It follows that $U^{*}\rho$ is positive and also note that $\text{Tr}(U^{*}\rho)=\text{Tr}(|e\rangle \langle e|)=\langle e, e\rangle=1.$ Hence $\rho $ is a quantum  $U$-state.
		\end{example}
		\begin{proposition}
			A quantum $(U_{A}, U_{B})$-channel $\psi :M_{A}\to M_{B}$ maps quantum $U_{A}$-states into quantum $U_{B}$-states.
		\end{proposition}
		\begin{proof}
			Let $V$ be a quantum $U_{A}$-state, that is, $V$ is $U_{A}$-positive and  $\text{Tr}(U_{A}^{*}V)=1.$ Since $\psi $ is  a quantum $(U_{A}, U_{B})$-channel, we have \begin{align*}
				\psi(V)=\sum_{i=1}^{l}R_{i}^{\#_{A,B}}VR_{i}=\sum_{i=1}^{l}U_B R^*_i U_{A}^{*}VR_{i},
			\end{align*}
			for some  $ m \times n$-matrices $R_1,\ldots,R_{l}.$ Since  $V$ is $U_{A}$-positive, we have $U_{A}^{*}V\geq 0.$ Therefore	$U_B^{*}\psi(V)=\sum_{i=1}^{l} R^*_i U_{A}^{*}VR_{i} \geq 0,$ that is, $\psi(V)$  is $U_{B}$-positive. Furthermore, we obtain 
			\begin{align*}
				\text{Tr}(U_B^{*}\psi(V))&= \text{Tr}(\sum_{i=1}^{l}R^*_i U_{A}^{*}VR_{i})=\text{Tr}(\sum_{i=1}^{l} U_{A}^{*}VR_{i}R^*_i)=\text{Tr}(U_{A}^{*}V\sum_{i=1}^{l}R_i R_{i}^*)\\&=\text{Tr}(U_{A}^{*}V)=1,
			\end{align*} which proves that $\psi(V)$ is a quantum $U_{B}$-state.
		\end{proof}
		
		A bipartite quantum state $\rho \in M_{A}\otimes M_{B}$ is a {\it product state} if $\rho =\rho_{A}\otimes \rho_{B}$ with $\rho_{A}\in M_{A}^{+}$ and $\rho_{B}\in M_{B}^{+}$ and  is {\it separable} if it is a convex combination of product states. Moreover, it is {\it entangled} if it is not separable. We define $\tau :=t \otimes \text{id}:M_{A}\otimes M_{B} \to M_{A}\otimes M_{B}$ where $t$ is the transpose on $M_{A}.$ We call the $\tau$ map the {\it partial transpose} or {\it the blockwise transpose} and a bipartite quantum state $\rho$ is {\it positive partial transpose} (PPT) if $\rho^{\tau}:=t \otimes \text{id}(\rho)$ is positive. The {\it positive partial transpose criterion} says that if $\rho$ is separable, then $\rho $ is positive partial transpose.
		\begin{definition}
			Let $U_{A}$ and $U_{B}$ be the fundamental unitaries in $M_{A}$ and $ M_{B},$ respectively. Let $U_{A}\otimes U_{B}$ be the fundamental unitary in $M_{A}\otimes M_{B}$ and $\rho \in M_{A}\otimes M_{B}$  be a bipartite quantum  $U_{A}\otimes U_{B}$-state. Then
			\begin{enumerate}
				\item $\rho$ is a {\rm product $U_{A}\otimes U_{B}$-state} if $\rho =\rho_{A}\otimes \rho_{B}$ where $\rho_{A}\in M_{A}^{U+}$ and $\rho_{B}\in M_{B}^{U+}.$
				\item  $\rho$ is {\rm $U_{A}\otimes U_{B}$-separable} if it is a convex combination of product  $U_{A}\otimes U_{B}$-states.
				\item  $\rho$ is {\rm $U_{A}\otimes U_{B}$-entangled} if it is not  $U_{A}\otimes U_{B}$-separable.
				\item $\rho$ is {\rm $U_{A}\otimes U_{B}$-positive partial transpose} if the partial transpose $\rho^{\tau}$ is  $U^t_{A}\otimes U_{B}$-positive, that is,  $(\overline{U}_{A}\otimes U_{B}^{*})(\rho^{\tau})$  is positive.
			\end{enumerate}
		\end{definition}
		\begin{proposition}
			If a  bipartite quantum  $U_{A}\otimes U_{B}$-state $\rho \in M_{A}\otimes M_{B}$ is $U_{A}\otimes U_{B}$-separable, then $\rho$ is $U_{A}\otimes U_{B}$-positive partial transpose.
		\end{proposition}
		\begin{proof}
			Consider that $\rho $   is $U_{A}\otimes U_{B}$-separable, it means  we can write it as a  convex combination of product  $U_{A}\otimes U_{B}$-states, that is,\begin{align*}
				\rho&=\sum_{i=1}^{l}p_i (U_{A}\otimes U_{B})(|z_{i} \rangle \langle z_{i}|)= \sum_{i=1}^{l}p_i (U_{A}\otimes U_{B})(|x_{i} \rangle \otimes |y_{i}\rangle ) (\langle x_{i}| \otimes \langle y_{i}| ) \\&= \sum_{i=1}^{l}p_i (U_{A}\otimes U_{B})(|x_{i} \rangle \langle x_{i}|  \otimes  |y_{i}\rangle  \langle y_{i}| ) = \sum_{i=1}^{l}p_i U_{A}(|x_{i} \rangle \langle x_{i}| ) \otimes  U_{B}(|y_{i}\rangle  \langle y_{i}| ),
			\end{align*} with $\sum_{i=1}^{l}p_i=1,$ and $|z_{i}\rangle =|x_{i} \rangle \otimes |y_{i}\rangle \in M_{A}\otimes M_{B} .$
			Since
			$(U_{A}(|x_{i} \rangle \langle x_{i}|))^{t}=|\overline{x_{i}} \rangle \langle \overline{x_{i}}|\overline{U_{A}^{*}},
			$ we obtain\begin{align*}
				\rho^{\tau}=t \otimes \text{id}(\rho)=\sum_{i=1}^{l}p_i |\overline{x_{i}} \rangle \langle \overline{x_{i}}|\overline{U_{A}^{*}}\otimes  U_{B}(|y_{i}\rangle  \langle y_{i}|).
			\end{align*}Since $\overline{U_{A}}|\overline{x_{i}} \rangle \langle \overline{x_{i}}|\overline{U_{A}^{*}}$ is a positive matrix in $M_{A},$ $(\overline{U}_{A}\otimes U_{B}^{*})(\rho^{\tau})$ is positive.
		\end{proof}
		\section{U-entanglement breaking maps}\label{sec5}
		In this section, we consider the special class of quantum channels which can be simulated by a classical channel in the following sense: The sender makes a measurement on the input state $\rho$, and send the outcome $k$ via a classical channel to the receiver who then prepares an agreed upon state $R_k.$ Such channels can be written in the form \begin{align*}
			\phi(\rho)=\sum_{k}R_{k}\text{Tr}(E_{k}\rho),
		\end{align*}where each $R_k$ is a {\it density matrix} (density matrices, also called density operators, which conceptually take the role of the state vectors, that is, $R_{k}$ is a positive semi-definite matrix with $\text{Tr}(R_{k})=1  $) and the ${E_k}$ form a positive operator valued measure ($\{E_k\}_{k}$ form a {\it positive operator valued measure} means for each $ k,$ $ E_{k}$ is positive semi-definite and $\sum_{k}E_{k}=id_{A}$). We call this the {\it “Holevo form”} because it was introduced by
		Holevo in \cite{H99}.
		In this context, it is natural to consider the class of channels which break entanglement.
		\begin{definition}Let $\phi: M_{A} \to M_{B}$ be  a quantum channel. If $(id_{n} \otimes \phi) (S)$ is always separable for all bipartite quantum states $S \in M_{n}(\mathbb{C})\otimes M_{A},$ then we call it an {\rm entanglement breaking map}. 	\end{definition}
		Let $U_{A}$ and $U_{B}$ be the fundamental unitaries in $M_{A}$ and $ M_{B},$ respectively. The family $ \{F_{k}\}_{k}$ is a {\it $U_{A}$-positive operator valued measure} if each $ F_{k}U_{A}$ is positive semi-definite and $\sum_{k}F_{k}U_{A}=id_{A}$ (or $\sum_{k}F_{k}=U_{A}^{*}$) and  $D$ is called {\it $U_{A}$-density matrix} if $D$ is a $U_{A}$-positive semi-definite matrix, that is, $U_{A}^{*}D$ is positive semi-definite matrix with $\text{Tr}(U_{A}^{*}D)=1. $
		\begin{definition}
			Let $\psi: M_{A} \to M_{B}$ be  a $(U_{A},U_{B})$-quantum channel. \begin{enumerate}
				\item $\psi$ is said to be {\rm $(U_{A},U_{B})$-entanglement breaking} if $(id_{n} \otimes\psi)(S)$ is $I_{n} \otimes U_{B}$-separable  for any $I_{n} \otimes U_{A}$-density matrix $S \in M_{n}(\mathbb{C})\otimes M_{A}.$
				\item $\psi$ is in {\rm $(U_{A},U_{B})$-Holevo form} if it can be expressed as 
				\begin{align*}
					\psi(\rho)=\sum_{k}D_{k}\text{Tr}(F_{k}\rho),
				\end{align*} where $D_{k}$ is a $U_{B}$-density matrix, that is, $U_{B}^{*}D_{k}$ is positive semi-definite matrix  and $\text{Tr}(U_{B}^{*}D_{k})=1$ and $F_{k}$ is  a $U_{A}$-positive operator valued measure in $M_{A},$ that is  $F_{k}U_{A}$ is positive semi-definite and $\sum_{k}F_{k}U_{A}=id_A.$
			\end{enumerate}
		\end{definition}
		
		\begin{theorem}\label{thrm}
			Let $\psi: M_{A} \to M_{B}$ be  a $(U_{A},U_{B})$-quantum channel. Then the following statements are equivalent:
			\begin{enumerate}
				\item  $\psi$ is  $(U_{A},U_{B})$-entanglement breaking;\item  $\psi$ is in $(U_{A},U_{B})$-Holevo form .
			\end{enumerate}
		\end{theorem}
		\begin{proof}	$(1)\implies (2):$ Suppose $\psi$ is  $(U_{A},U_{B})$-entanglement breaking. The map $\phi$ given by $\phi(V)=U^*_B \psi(U_AV)$ is a quantum channel and we have for each $ n \in \mathbb{N},$
			\begin{align}\label{en}
				id_{n}\otimes \phi=	id_{n}\otimes (U^*_B \psi(U_A))=(I_{n}\otimes U^*_B)(	id_{n}\otimes \psi)(I_{n}\otimes U_{A}).
			\end{align}Let $S\in M_{n}(\mathbb{C}) \otimes M_{A}$ be a  density matrix. One can easily verify that  $(I_{n}\otimes U_{A})S$ is a $(I_{n}\otimes U_{A})$-density matrix, that is, $(I_{n}\otimes U_{A}^{*})(I_{n}\otimes U_{A})S$ is positive and $\text{Tr}((I_{n}\otimes U_{A}^{*})(I_{n}\otimes U_{A})S)=1$ which trivially hold as $
			(I_{n}\otimes U_{A}^{*})(I_{n}\otimes U_{A})S=S.$ Since $(	id_{n}\otimes \psi)(I_{n}\otimes U_{A})S$ is $(I_{n}\otimes U_B)$-separable, $	(id_{n}\otimes \phi)(S)$ is separable. This implies that $\phi$ is an entanglement breaking map. Now using \cite[ Theorem 4]{HPM03},  we can write $\phi$ in the Holevo form, that is, \begin{align*}
				\phi(\rho)=\sum_{k}R_{k}\text{Tr}(E_{k}\rho),
			\end{align*}
			where each $R_{k}$ is a density matrix and $\{E_{k}\}_{k}$ is a positive operator valued measure with $\sum_{k}E_{k}=id_A.$ Observe that \begin{align*}
				\psi(\rho)=U_B \phi(U_{A}^{*}\rho)= \sum_{k}U_{B}R_{k}\text{Tr}(E_{k}U_{A}^{*}\rho)=\sum_{k}D_{k}\text{Tr}(F_{k}\rho),
			\end{align*} where $D_{k}:=U_{B}R_{k}$ and $F_{k}:=E_{k}U_{A}^{*}.$ Note  that $D_{k}$ is a $U_{B}$-density matrix since $U_{B}^{*}D_{k}=U_{B}^{*}U_{B}R_{k}=R_{k}$ and $R_{k}$ is already a density matrix in $M_{B}$ and also $\{F_{k}\}_k$ is a $U_{A}$-positive operator valued measure in $M_{A}$ as $E_{k}U_{A}^{*}U_{A}=E_{k}$ is positive semi-definite and $\sum_{k}E_{k}U_{A}^{*}U_{A}=id_A.$

			$(2)\implies (1):$ Assume that $\psi$ has the $(U_{A},U_{B})$-Holevo form, it means $\psi(\rho)=\sum_{k}D_{k}\text{Tr}(F_{k}\rho),$ where  $D_{k}$ is a $U_{B}$-density matrix and $\{F_{k}\}_{k}$ is a $U_{A}$-positive operator valued measure in $M_{A}.$ Define $\phi$ by $\phi(\rho)=U^*_B \psi(U_A\rho),$ where $\rho\in M_{A}.$ We obtain \begin{align*}
				\phi(\rho)= U^*_B \psi(U_A\rho)=U^*_B \psi(U_A\rho)&=U^*_B \sum_{k}D_{k}\text{Tr}(F_{k}U_A\rho)\\&= \sum_{k} U^*_BD_{k}\text{Tr}(F_{k}U_A\rho).
			\end{align*}Since $D_{k}$ is a $U_{B}$-density matrix and $\{F_{k}\}_{k}$ is a $U_{A}$-positive operator valued measure in $M_{A},$ $\phi$ has a Holevo form and by \cite[ Theorem 4]{HPM03} $\phi$ is an entanglement breaking map and hence Equation (\ref{en}) implies that $\psi$ is a $(U_{A},U_{B})$-entanglement breaking map.
		\end{proof}
		\begin{remark} Let $\phi,\psi: M_{A} \to M_{B}$ be  linear maps such that
			$\psi(\rho)=U_B \phi(U_{A}^{*}\rho),$ where $\rho \in M_{A} .$ As we know $\phi$ is positive if and only if $\psi$ is a  $(U_{A},U_{B})$-positive map. Suppose $\phi$ is a quantum channel, that is, $\psi$ is a  $(U_{A},U_{B})$-quantum channel. Note that $\theta\circ\phi$ is a CP map for any CP map $\theta : M_{B} \to M_{C}$ if and only if $\omega\circ \psi$ is $(U_{A},U_{C})$-CP  for any $(U_{B},U_{C})$-CP $\omega: M_{B} \to M_{C}.$ Therefore, it follows from Theorem \ref{thrm} that $\phi$ is an entanglement breaking map if and only if $\psi$ is a $(U_{A},U_{B})$-entanglement breaking map.
		\end{remark}
		\section{Examples of fundamental unitary and $U$-CP maps}\label{sec6}
		In this section, we provide concrete examples of  completely $U$-positive maps and examples of $3 \otimes 3$ quantum $U$-states which are $U$-entangled and $U$-separable.	It is easy to observe that the $2\times 2$ identity matrix $I$ and the Pauli matrices
		\[\sigma_x=\begin{pmatrix}
			0 & 1\\
			1 & 0
		\end{pmatrix},\quad\sigma_y=\begin{pmatrix}
			0 & -\iota\\
			\iota & 0
		\end{pmatrix},\quad\sigma_z=\begin{pmatrix}
			1 & 0\\
			0 & -1
		\end{pmatrix}\]
		form a basis for $M_2(\mathbb{C})$. That is, for any $A\in M_2(\mathbb{C}),$ we have $A=aI+b\sigma_x+c\sigma_y+d\sigma_z$ where $a,b,c,d\in\mathbb{C}$.
		Any fundamental unitary on the $2$-dimensional complex $S$-space has the form
		\begin{equation}
			U=\begin{pmatrix}\label{mU}
				a & b\\
				-e^{\iota\phi}\overline{b} & e^{\iota\phi}\overline{a}
			\end{pmatrix}
		\end{equation}
		where $\phi\in\mathbb{R}$ and $a,b\in\mathbb{C}$ such that $|a|^{2}+|b|^{2}=1.$ For example, if we choose  $a=1$ and $b=0$, then we have the unitary
		\begin{equation*}\label{U}
			\begin{pmatrix}
				1 & 0\\
				0 & e^{\iota\phi}
			\end{pmatrix}
		\end{equation*}
		which is called a {\rm Phase Gate} (see \cite{qc}) that represents a rotation about the $z$-axis by an
		angle $\phi$ on the Bloch sphere.

		If we define an $S$-space with respect to the fundamental unitary $U$ as in \eqref{mU}, then $U^{*}A=aU^{*}+b\sigma_x^{U}+c\sigma_y^{U}+d\sigma_z^{U},$ where $\sigma_x^{U}=U^{*}\sigma_x,$ $\sigma_y^{U}=U^{*}\sigma_y,$ and $\sigma_z^{U}=U^{*}\sigma_z,$ and we call these matrices {\it $U$-Pauli matrices}.

		Let $U_1=\begin{pmatrix}
			1 & 0\\
			0 & \iota
		\end{pmatrix}$ and $U_2=\frac{1}{\sqrt{2}}\begin{pmatrix}
			1 & -1\\
			1 & 1
		\end{pmatrix}$ be two unitaries which are not symmetries, where $U_{1}$ is the Phase gate for $\phi=\frac{\pi}{2}.$ \begin{enumerate}
			\item  Consider the $S$-space $(\mathbb{C}^2, U_1).$ For any $A\in M_2(\mathbb{C}),$ we have
			\[U^{*}_1A=\left[a\begin{pmatrix}
				1 & 0\\
				0 & -\iota
			\end{pmatrix}-\iota b\begin{pmatrix}
				0 & \iota\\
				1 & 0
			\end{pmatrix}-\iota c\begin{pmatrix}
				0 & 1\\
				\iota & 0
			\end{pmatrix}+d\begin{pmatrix}
				1 & 0\\
				0& \iota
			\end{pmatrix}\right]\] and
			\[(U^{*}_1A)^{*}=\left[\overline{a}\begin{pmatrix}
				1 & 0\\
				0 & \iota
			\end{pmatrix}+\overline{b}\begin{pmatrix}
				0 & \iota\\
				1 & 0
			\end{pmatrix}+ \overline{c}\begin{pmatrix}
				0 & 1\\
				\iota& 0
			\end{pmatrix}+\overline{d}\begin{pmatrix}
				1 & 0\\
				0& -\iota
			\end{pmatrix}\right].\]  
			Comparing $U^{*}_1A$ and $(U^{*}_1A)^*$, one may easily find out that $A$ is $U_1$-self adjoint if and only if $ a=\overline{d},$ $-\iota c=\overline{c}$ and $-\iota b=\overline{b}$ , that is, $A$ has the form
			\begin{align*}A=
				\begin{pmatrix}
					a+d & b-\iota c\\
					b+\iota c & a-d
				\end{pmatrix}=	\begin{pmatrix}
					a+\overline{a} & b+\overline{c}\\
					b-\overline{c} & a-\overline{a}
				\end{pmatrix}=\begin{pmatrix}
					2 \Re(a) & b+\overline{c}\\
					b-\overline{c} & 2\iota\Im(a)
				\end{pmatrix}
			\end{align*} and  $U_{1}^{*}A$ has the form
			\begin{align*}
				U_{1}^{*}A=
				\begin{pmatrix}
					a+d & b-\iota c\\
					c-\iota b & -\iota(a-d)
				\end{pmatrix}=	\begin{pmatrix}
					a+\overline{a} & b+\overline{c}\\
					c+\overline{b} & \iota(a-\overline{a})
				\end{pmatrix}=\begin{pmatrix}
					2 \Re(a) & b+\overline{c}\\
					c+\overline{b} & 2\Im(a)
				\end{pmatrix}
			\end{align*}
			where $a,b,c\in\mathbb{C}$. Further, $U^{*}_1A$ is positive, that is, $A$ is $U_1$-positive if and only if
			\[0\leq \Re(a)\quad\mbox{and}\quad 4\Re(a)\Im(a)\geq  (b+\overline{c})(\overline{b}+{c})\]
			Also, $U^{*}_1A$ is a quantum state, that is, $A$ is a quantum $U_1$-state if and only if
			\[\Re(a)+\Im(a)=\frac{1}{2}.\]
			In particular, if $a=\frac{1}{2}\in\mathbb{R},$ $b= t$ and $c= -t$ for all $t \geq 0$, then all the above relations are trivially satisfied. In other words, for $t \geq 0$,
			\[A=\rho_t=\begin{pmatrix}
				1 & 0\\
				2 t & 0
			\end{pmatrix}\]
			provides a one parameter family of quantum $U_1$-states in $M_2(\mathbb{C})$. Similarly, the following provides a one parameter family of quantum $U_1\otimes U_1$-states
			\[\frac{1}{16}\begin{pmatrix}
				1 & 0 &0 & 0\\
				2 t & 0 & 0 &0\\
				2 t & 0& 0 & 0\\
				4t^{2} & 0 & 0 & 0
			\end{pmatrix},\]
			where $t\geq 0$.

			Since $M_{2}(\mathbb{C})$ is a unital $*$-algebra, any $*$-homomorphism $\pi$ from $M_{2}(\mathbb{C})$ into $M_{2}(\mathbb{C})$  has the form $ \pi(A)=W^{*}AW $ for some unitary matrix $W \in M_{2}(\mathbb{C}).$ If $\phi$ is a  $U_1$-CP map defined on $M_{2}(\mathbb{C}),$ then
			by Theorem \ref{2.2} there exist a $*$-homomorphism $\pi$ on $M_2(\mathbb{C})$ and a matrix $V \in M_{2}(\mathbb{C}) $  such that  $$\phi (A)=V^{\#}\pi(A)V, $$ where $V^{\#}=U_1V^{*}U_1^{* }.$	For example, if we consider $V=\begin{pmatrix}
				\alpha & 0\\
				0 & \beta
			\end{pmatrix}$ and a unitary $W= \begin{pmatrix}
				\gamma & 0\\
				0 & \delta
			\end{pmatrix},$ then  we  get $U_1$-CP $\phi$ in the following form:
			\begin{align*}
				\phi(A)&=V^{\#}\pi(A)V=(U_1V^{*}U_1^{*})(W^{*}AW )V=\begin{pmatrix}
					\overline{\alpha}&0\\
					0 &\overline{\beta}
				\end{pmatrix}\begin{pmatrix}
					a_{11}& \overline{\gamma}a_{12}\delta\\
					\overline{\delta}a_{21}\gamma& a_{22}
				\end{pmatrix}\begin{pmatrix}
					\alpha & 0\\
					0 & \beta
				\end{pmatrix}\\&= \begin{pmatrix}
					\overline{\alpha}\alpha a_{11} & \overline{\alpha}\overline{\gamma}\delta\beta a_{12} \\
					\overline{\beta}\overline{\delta}\gamma\alpha a_{21} &\overline{\beta}\beta a_{22}
				\end{pmatrix},
			\end{align*} where $A =\begin{pmatrix}
				a_{11} & a_{12}\\
				a_{21} & a_{22}
			\end{pmatrix} \in M_{2}(\mathbb{C})$. Furthermore, if $|\alpha|=|\beta|=1,$ then $\phi(A)$ is of the form
			
			\begin{equation*}
				\phi(A)=\begin{pmatrix}
					a_{11} & \overline{\alpha}\overline{\gamma}\delta\beta a_{12} \\
					\overline{\beta}\overline{\delta}\gamma\alpha a_{21} & a_{22}
				\end{pmatrix}.
			\end{equation*}

			\item Consider the $S$-space $(\mathbb{C}^2, U_2).$
			For any $A\in M_2(\mathbb{C}),$ we obtain
			\[U^{*}_2A=\frac{1}{\sqrt{2}}\left[a\begin{pmatrix}
				1 & 1\\
				-1 & 1
			\end{pmatrix}+b\begin{pmatrix}
				1 & 1\\
				1 & -1
			\end{pmatrix}-\iota c\begin{pmatrix}
				-1 & 1\\
				-1 & -1
			\end{pmatrix}+d\begin{pmatrix}
				1 & -1\\
				-1 & -1
			\end{pmatrix}\right].\]
			Comparing $U^{*}_2A$ and $(U^{*}_2A)^*$, one may easily find out that $A$ is $U_2$-self adjoint if and only if $b$ and $d$ are reals and $c=-\iota\overline{a}$, that is, $A$ has the form
			\[\begin{pmatrix}
				a+d & -\overline{a}+b\\
				\overline{a}+b & a-d
			\end{pmatrix}\]
			where $a\in\mathbb{C}$ and $b,d\in\mathbb{R}$. Further, $U^{*}_2A$ is positive, that is, $A$ is $U_2$-positive if and only if
			\[-(b+d)\leq 2\Re(a)\quad\mbox{and}\quad b^2+d^2\leq 2((\Re(a))^2-(\Im(a))^2).\]
			Also, $U^{*}_2A$ is a quantum state, that is, $A$ is a quantum $U_2$-state if and only if
			\[\Re(a)=\frac{\sqrt{2}}{4},\quad-(b+d)\leq\frac{\sqrt{2}}{2}\quad\mbox{and}\quad b^2+d^2\leq\frac{1}{4}-2(\Im(a))^2.\]
			In particular, if $a=\sqrt{2}/4\in\mathbb{R}$ and $b=d=t/4$, with $-\sqrt{2}\leq t\leq\sqrt{2}$, then all the above relations are trivially satisfied. In other words, for $-\sqrt{2}\leq t\leq\sqrt{2}$,
			\[\rho_t=\frac{1}{4}\begin{pmatrix}
				t+\sqrt{2} & t-\sqrt{2}\\
				t+\sqrt{2} & -t+\sqrt{2}
			\end{pmatrix}\]
			provides a one parameter family of quantum $U_2$-states in $M_2(\mathbb{C})$. Similarly, the following provides a one parameter family of quantum $U_2\otimes U_2$-states
			\[\frac{1}{16}\begin{pmatrix}
				t^2+2\sqrt{2}t+2 & t^2-2 & t^2-2 & t^2-2\sqrt{2}t+2\\
				t^2+2\sqrt{2}t+2 & -t^2+2 & t^2-2 & -t^2+2\sqrt{2}t-2\\
				t^2+2\sqrt{2}t+2 & t^2-2 & -t^2+2 & -t^2+2\sqrt{2}t-2\\
				t^2+2\sqrt{2}t+2 & -t^2+2 & -t^2+2 & t^2-2\sqrt{2}t+2
			\end{pmatrix},\]
			where $-\sqrt{2}\leq t\leq\sqrt{2}$.

			Also, similar to the earlier example, we  get any $U_2$-CP map $\phi$ in the following form:
			\begin{align*}
				\phi(A)&=V^{\#}\pi(A)V=(U_2V^{*}U_2^{*})(W^{*}AW )V\\&=\frac{1}{{2}}\begin{pmatrix}
					\overline{\alpha}+\overline{\beta} & \overline{\alpha}-\overline{\beta}\\
					\overline{\alpha}-\overline{\beta} & \overline{\alpha}+\overline{\beta}
				\end{pmatrix}\begin{pmatrix}
					a_{11}& \overline{\gamma}a_{12}\delta\\
					\overline{\delta}a_{21}\gamma& a_{22}
				\end{pmatrix}\begin{pmatrix}
					\alpha & 0\\
					0 & \beta
				\end{pmatrix}\\&= \frac{1}{{2}}\begin{pmatrix}
					(\overline{\alpha}+\overline{\beta})\alpha a_{11}+(\overline{\alpha}-\overline{\beta})\overline{\delta}\gamma\alpha a_{21} & (\overline{\alpha}+\overline{\beta})\overline{\gamma}\delta\beta a_{12} +(\overline{\alpha}-\overline{\beta})\beta a_{22}\\
					(\overline{\alpha}-\overline{\beta})\alpha a_{11}+(\overline{\alpha}+\overline{\beta})\overline{\delta}\gamma\alpha a_{21} & (\overline{\alpha}-\overline{\beta})\overline{\gamma}\delta\beta a_{12} +(\overline{\alpha}+\overline{\beta})\beta a_{22}
				\end{pmatrix},
			\end{align*} where $A =\begin{pmatrix}
				a_{11} & a_{12}\\
				a_{21} & a_{22}
			\end{pmatrix} \in M_{2}(\mathbb{C})$. Also if $|\alpha|=|\beta|=1,$ then $\phi(A)$ is of the form
			
			\begin{equation*}
				\phi(A)=\frac{1}{{2}}\begin{pmatrix}
					(1+\overline{\beta}\alpha) a_{11}+(1-\overline{\beta}\alpha)\overline{\delta}\gamma a_{21} & (\overline{\alpha}\beta+1)\overline{\gamma}\delta a_{12} +(\overline{\alpha}\beta -1)a_{22}\\
					(1-\overline{\beta}\alpha) a_{11}+(1+\overline{\beta}\alpha)\overline{\delta}\gamma a_{21} & (\overline{\alpha}\beta-1)\overline{\gamma}\delta a_{12} +(\overline{\alpha}\beta+1) a_{22}
				\end{pmatrix}.
			\end{equation*}

			\item Let $\mathbb{C}^{3}$ be a $3$-dimensional $S$-space with an indefinite metric induced by $U_3$, where
			$U_3=\frac{1}{\sqrt{2}}\begin{pmatrix}
				1 & -1 & 0\\
				1 & 1 & 0\\
				0 & 0 & \sqrt{2}
			\end{pmatrix}.$
			It is easy to observe that the matrices
			\begin{align*}
				&\mu_1=\begin{pmatrix}
					1 & 0 & 0\\
					1 & 0 & 0\\
					0 & 0 & 0
				\end{pmatrix},\quad
				\mu_2=\begin{pmatrix}
					0 & 1 & 0\\
					0 & 1 & 0\\
					0 & 0 & 0
				\end{pmatrix},\quad
				\mu_3=\begin{pmatrix}
					0 & 0 & 1\\
					0 & 0 & 1\\
					0 & 0 & 0
				\end{pmatrix},\\
				&\mu_4=\begin{pmatrix}
					-1 & 0 & 0\\
					1 & 0 & 0\\
					0 & 0 & 0
				\end{pmatrix},\quad
				\mu_5=\begin{pmatrix}
					0 & -1 & 0\\
					0 & 1 & 0\\
					0 & 0 & 0
				\end{pmatrix},\quad
				\mu_6=\begin{pmatrix}
					0 & 0 & -1\\
					0 & 0 & 1\\
					0 & 0 & 0
				\end{pmatrix},\\
				&\mu_7=\begin{pmatrix}
					0 & 0 & 0\\
					0 & 0 & 0\\
					\sqrt{2} & 0 & 0
				\end{pmatrix},\quad
				\mu_8=\begin{pmatrix}
					0 & 0 & 0\\
					0 & 0 & 0\\
					0 & \sqrt{2} & 0
				\end{pmatrix},\quad
				\mu_9=\begin{pmatrix}
					0 & 0 & 0\\
					0 & 0 & 0\\
					0 & 0 & \sqrt{2}
				\end{pmatrix}
			\end{align*}
			form a basis for $M_3(\mathbb{C})$. Thus, for any $A\in M_3(\mathbb{C})$, we have $A=\sum\limits_{i=1}^{9}a_i\mu_{i},$ where $a_i\in\mathbb{C}$. Then, we get
			\begin{align}\label{mat}
				A=\begin{pmatrix}
					a_1-a_4 & a_2-a_5 & a_3-a_6\\
					a_1+a_4 & a_2+a_5 & a_3+a_6\\
					a_7\sqrt{2} & a_8\sqrt{2} & a_9\sqrt{2}
				\end{pmatrix}.
			\end{align}
			Since
			\begin{align*}
				U_3^{*}A=\sqrt{2}\begin{pmatrix}
					a_1 & a_2 & a_3\\
					a_4 & a_5 & a_6\\
					a_7 & a_8 & a_9
				\end{pmatrix},
			\end{align*}
			after comparing $U_3^{*}A$ and $(U_3^{*}A)^*$, one may easily find out that $A$ is $U_3$-self adjoint if and only if $a_1,$ $a_5$ and $a_9$ are reals and $a_2=\overline{a_4},$ $a_3=\overline{a_7}$ and $a_6=\overline{a_8}$, that is, $U_3^{*}A$ has the form
			\[U_3^{*}A=\sqrt{2}\begin{pmatrix}
				a_1 & a_2 & a_3\\
				\overline{a_2} & a_5 & a_6\\
				\overline{a_3} & \overline{a_6} & a_9
			\end{pmatrix}.\]
			Further, $U_3^{*}A$ is positive, that is, $A$ is $U_3$-positive if and only if the following conditions hold:
			\begin{align}
				&a_1\geq 0,\\
				&a_1a_5-|a_2|^{2}\geq 0\\
				\mbox{and}\quad &a_1a_5a_9-a_1|a_6|^{2}-|a_2|^{2}a_9-|a_3|^2a_5+ 2\Re(a_2\overline{a_3}a_6)\geq 0.
			\end{align}
			Also, $U_3^{*}A$ is a quantum state, that is, $A$ is a quantum $U_3$-state if and only if
			\[a_1+a_5+a_9=\frac{1}{\sqrt{2}}.\]
			In particular, if we choose $a_i=\dfrac{1}{3\sqrt{2}}$ in \eqref{mat}, then the matrix
			$A=\dfrac{1}{3}\begin{pmatrix}
				0 & 0 & 0\\
				\sqrt{2} & \sqrt{2} & \sqrt{2}\\
				1 & 1 &1
			\end{pmatrix}$
			is a $U_3$-state, where
			\begin{align*}
				U_3^{*}A=\frac{1}{3}\begin{pmatrix}
					1 & 1 & 1\\
					1 & 1 & 1\\
					1 & 1 & 1
				\end{pmatrix}.
			\end{align*}
			Using this example we give the following quantum separable $U_3\otimes U_3$-state:
			\begin{align*}
				\frac{1}{9}\begin{pmatrix}
					0 & 0 & 0 & 0 & 0 & 0 & 0 & 0 & 0\\
					0 & 0 & 0 & 0 & 0 & 0 & 0 & 0 & 0\\
					0 & 0 & 0 & 0 & 0 & 0 & 0 & 0 & 0\\
					0 & 0 & 0 & 0 & 0 & 0 & 0 & 0 & 0\\
					2 & 2 & 2 & 2 & 2 & 2 & 2 & 2 & 2\\
					\sqrt{2} & \sqrt{2} & \sqrt{2} & \sqrt{2} & \sqrt{2} & \sqrt{2} & \sqrt{2} & \sqrt{2} & \sqrt{2}\\
					0 & 0 & 0 & 0 & 0 & 0 & 0 & 0 & 0\\
					\sqrt{2} & \sqrt{2} & \sqrt{2} & \sqrt{2} & \sqrt{2} & \sqrt{2} & \sqrt{2} & \sqrt{2} & \sqrt{2}\\
					1 & 1 & 1 & 1 & 1 & 1 & 1 & 1 & 1
				\end{pmatrix}.
			\end{align*}
			In \cite{choi 2}, Choi gave the following entangled state which has positive partial transpose:
			\begin{align*}
				\begin{pmatrix}
					1&0&0&0&1&0&0&0&1\\
					0&2&0&1&0&0&0&0&0\\
					0&0&\frac{1}{2}&0&0&0&1&0&0\\
					0&1&0&\frac{1}{2}&0&0&0&0&0\\
					1&0&0&0&1&0&0&0&1\\
					0&0&0&0&0&2&0&1&0\\
					0&0&1&0&0&0&2&0&0\\
					0&0&0&0&0&1&0&\frac{1}{2}&0\\
					1&0&0&0&1&0&0&0&1
				\end{pmatrix}.
			\end{align*}
			Consider
			\begin{align*}
				C:=\frac{2}{21}\begin{pmatrix}
					1&0&0&0&1&0&0&0&1\\
					0&2&0&1&0&0&0&0&0\\
					0&0&\frac{1}{2}&0&0&0&1&0&0\\
					0&1&0&\frac{1}{2}&0&0&0&0&0\\
					1&0&0&0&1&0&0&0&1\\
					0&0&0&0&0&2&0&1&0\\
					0&0&1&0&0&0&2&0&0\\
					0&0&0&0&0&1&0&\frac{1}{2}&0\\
					1&0&0&0&1&0&0&0&1
				\end{pmatrix}.
			\end{align*}
			Note that
			\begin{align*}
				U_3\otimes U_3=\frac{1}{2}\begin{pmatrix}
					1 & -1 & 0 & -1 & 1 & 0 & 0 & 0 & 0\\
					1 & 1 & 0 & -1 & -1 & 0 & 0 & 0 & 0\\
					0 & 0 & \sqrt{2} & 0 & 0 & -\sqrt{2} & 0 & 0 & 0\\
					1 & -1 & 0 & 1 & -1 & 0 & 0 & 0 & 0\\
					1 & 1 & 0 & 1 & 1 & 0 & 0 & 0 & 0\\
					0 & 0 & \sqrt{2} & 0 & 0 & \sqrt{2} & 0 & 0 & 0\\
					0 & 0 & 0 & 0 & 0 & 0 & \sqrt{2} & -\sqrt{2} & 0\\
					0 & 0 & 0 & 0 & 0 & 0 & \sqrt{2} & \sqrt{2} & 0\\
					0 & 0 & 0 & 0 & 0 & 0 & 0 & 0 & 2
				\end{pmatrix}.
			\end{align*}
			One may easily check that
			\begin{align*}
				A=(U_3\otimes U_3)C=\frac{1}{21}\begin{pmatrix}
					2 & -3 & 0 & -\frac{3}{2} & 2 & 0 & 0 & 0 & 2\\
					0 & 1 & 0 & \frac{1}{2} & 0 & 0 & 0 & 0 & 0\\
					0 & 0 & \frac{1}{\sqrt{2}} & 0 & 0 & -2\sqrt{2} & \sqrt{2} & -\sqrt{2} & 0\\
					0 & -1 & 0 & -\frac{1}{2} & 0 & 0 & 0 & 0 & 0\\
					2 & 3 & 0 & \frac{3}{2} & 2 & 0 & 0 & 0 & 2\\
					0 & 0 & \frac{1}{\sqrt{2}} & 0 & 0 & 2\sqrt{2} & \sqrt{2} & \sqrt{2} & 0\\
					0 & 0 & \sqrt{2} & 0 & 0 & -\sqrt{2} & 2\sqrt{2} & -\frac{1}{\sqrt{2}} & 0\\
					0 & 0 & \sqrt{2} & 0 & 0 & \sqrt{2} & 2\sqrt{2} & \frac{1}{\sqrt{2}} & 0\\
					2 & 0 & 0 & 0 & 2 & 0 & 0 & 0 & 2
				\end{pmatrix}
			\end{align*}
			is a $U_3\otimes U_3$-entangled state.
		\end{enumerate}

		\subsection*{Acknowledgement}
		Azad Rohilla thanks Harsh Trivedi for a research visit during February-March 2024 at The LNM Institute of Information Technology.  Harsh Trivedi is supported by MATRICS-SERB  
		Research Grant, File No: MTR/2021/000286, by 
		SERB, Department of Science \& Technology (DST), Government of India.   Harsh Trivedi acknowledges the DST-FIST program (Govt. of India) FIST - No. SR/FST/MS-I/2018/24.

	\end{document}